\documentclass{article}
\usepackage[utf8]{inputenc}
\usepackage{amsmath}
\usepackage{amssymb}
\usepackage{amsthm}
\usepackage{authblk}
\usepackage{appendix}
\usepackage{comment}
\usepackage{subfigure}
\usepackage{adjustbox}
\usepackage[numbers]{natbib}
\usepackage{graphicx}

\title{Two-outcome synchronous correlation sets and Connes' embedding problem}
\author{Travis B. Russell}
\affil{Army Cyber Institute, \\ United States Military Academy, \\ West Point, NY}
\affil{\textit{travis.russell@westpoint.edu}}
\date{}

\newtheorem{theorem}{Theorem}[section]
\newtheorem{lemma}[theorem]{Lemma}
\newtheorem{proposition}[theorem]{Proposition}
\newtheorem{corollary}[theorem]{Corollary}
\newtheorem{definition}[theorem]{Definition}

\newtheorem{remark}[theorem]{Remark}

\newtheorem{problem}[theorem]{Problem}

\newcommand*\Span{\text{span}}

\begin{document}

\maketitle

\begin{abstract} We show that Connes' embedding problem is equivalent to the weak Tsirelson problem in the setting of two-outcome synchronous correlation sets. We further show that the extreme points of two-outcome synchronous correlation sets can be realized using a certain class of universal C*-algebras. We examine these algebras in the three-experiment case and verify that the strong and weak Tsirelson problems have affirmative answers in that setting. \end{abstract}

\section{Introduction}

Connes' embedding problem asks whether each $II_1$-factor acting on a separable Hilbert space admits a trace-preserving embedding into $\mathcal{R}^\omega$, the ultraproduct of the hyperfinite $II_1$ factor $\mathcal{R}$ \cite{ConnesConjecture}. In spite of decades of interest and many equivalent reformulations, this problem remains unsovled. 

Connes' problem has received renewed interest in light of a surprising reformulation in terms of Tsirelson's problems in quantum information theory. The weak Tsirelson problem asks whether the closure of the set of quantum correlations $C_q(n,m)$, a convex subset of $\mathbb{R}^{n^2 m^2}$, is equal to the set of quantum-commuting correlations $C_{qc}(n,m)$, a closed convex subset of $\mathbb{R}^{n^2 m^2}$, for every pair of positive integers $n$ and $m$. The connection to Connes' problem was first observed independently by Junge et. al. \cite{JMPPSW2011} and Fritz \cite{FritzKirchberg}. Ultimately Ozawa \cite{MR3067294} proved that Connes' problem is equivalent to the weak Tsirelson problem. Tsirelson's problem has important ramifications in quantum theory, as it essentially asks whether the quantum mechanics of Bohr-Heisenberg produces the same experimental data as the relativistic quantum mechanics of Haag-Kastler in the bipartite setting \cite{JMPPSW2011}. Moreover an understanding of the geometry of Tsirelson's correlation sets could have important practical significance for quantum communication and quantum computing \cite{PhysRevA.97.022104}.

As a byproduct of these connections, many authors from the mathematical community have invested their attention in Tsirelson's problems recently. Of special interest to us will be a series of papers related to synchronous quantum correlations sets. The synchronous quantum correlation sets are affine slices of the original correlation sets of Tsirelson. The study of synchronous quantum correlation sets was initiated in \cite{MR3460238} which produced a representation theorem for these correlations in terms tracial states on C*-algebras. Subsequently, Dykema-Paulsen proved in \cite{DykemaPaulsen2016} that Connes' embedding problem is equivalent to the weak Tsirelson problem restricted to the setting of synchronous correlation sets. In this paper we refine this result by showing that Connes' embedding problem is equivalent to the weak Tsirelson problem in the setting of two-outcome synchronous correlation sets. This result implies a recent unpublished result of Ozawa.

In another paper, Dykema-Paulsen-Parakash \cite{DykemaPaulsenPrakash2019} provided a negative answer to the strong Tsirelson problem (whether or not $C_q(n,m) = C_{qc}(n,m)$) in the setting of synchronous correlation sets with $n=5$ and $m=2$. In \cite{SynchronousGeometry} the author determined the geometry of the synchronous quantum correlation set in the setting of $n=3$ and $m=2$ but left the solution to Tsirelson's problems in this setting open. In this paper we complete the study of the $n=3, m=2$ case by showing that the strong Tsirelson problem has an affirmative answer in this setting. Along the way we show that Connes' problem is equivalent to the coincidence of a certain family of non-linear functionals defined on the two-outcome synchronous quantum correlation sets. Using this reformulation, we show that extreme points in the two-outcome synchronous quantum correlations can be realized with certain universal C*-algebras. We consider the structure of these C*-algebras in the three-experiment setting. All of these results build upon techniques used by Dykema-Paulsen-Prakash in \cite{DeltaGame} and the geometric simplifications of correlation sets which are implicit in Musat-R{\o}rdam \cite{Musat2019}.

Our paper is organized as follows. In section 2 we review the necessary notations and definitions and any theorems from the literature that we will need. In section 3 we establish the equivalence of Connes' conjecture with the weak Tsirelsen problem in the setting of two-outcome synchronous correlations. In section 4 we define a family of non-linear functionals. By analyzing these functionals, we obtain another equivalence with Connes' problem as well as our results concerning the three-experiment case in particular and extreme points of synchronous correlation sets in general.

\section{Preliminaries}

We first establish some notation. Let $\mathbb{N}$ denote the set of positive integers, $\mathbb{R}$ denote the set of real numbers, and $\mathbb{C}$ denote the set of complex numbers. For each $n \in \mathbb{N}$ we let $[n] = \{0,1,\dots,n-1\}$. 

By a \textbf{C*-algebra} we mean a unital self-adjoint closed subalgebra of bounded operators on a complex Hilbert space. We use basic results concerning C*-algebras freely throughout. We refer readers not familiar with the theory of C*-algebras to \cite{MR1402012}. By a \textbf{projection} we mean an operator $P$ satisfying $P^*=P$ and $P^2 = P$. By a \textbf{projection-valued measure}, we mean a set $\{P_i\}_{i \in [m]}$ of projections satisfying $\sum_{i \in [m]} P_i = I$, were $I$ denotes the identity operator. We note that when $\{P_i\}$ is a projection-valued measure it is necessarily the case that $P_i P_j = 0$ for all $i \neq j$. Given a C*-algebra $\mathfrak{A}$, we call a linear map $\phi: \mathfrak{A} \rightarrow \mathbb{C}$ a \textbf{state} if $\phi(I) = 1$ and $\phi(x^*x) \geq 0$ for all $x \in \mathfrak{A}$. 

We can now define the quantum correlation sets. Some of these definitions are non-standard, but an application of the GNS construction shows them to be equivalent to the standard ones.

\begin{definition} Let $n,m \in \mathbb{N}$. By a \textbf{correlation} we mean a tensor \[ \{ p(i,j|x,y) \}_{x,y \in [n], i,j \in [m]} \in \mathbb{R}^{n^2 m^2} \] satisfying $\sum_{i,j \in [m]} p(i,j|x,y) = 1$ for each $x,y \in [n]$. The set of all correlations is denoted by $C(n,m)$. We define the non-signaling, quantum-commuting, quantum, quantum-approximate and local correlations as follows. \begin{enumerate}
    \item A correlation $p \in C(n,m)$ is \textbf{non-signaling} if the quantities \[ p_A(i|x) = \sum_{j \in [m]} p(i,j|x,y), \quad p_B(j|y) = \sum_{i \in [m]} p(i,j|x,y) \] are well-defined (i.e. $p_A$ is independent of the choice of $y$ and $p_B$ is independent of the choice of $x$). The matrices $p_A(i|x)$ and $p_B(j|y)$ are called the marginal densities for $p$. The set of all non-signaling correlations is denoted by $C_{ns}(n,m)$.
    \item A correlation $p \in C(n,m)$ is called \textbf{quantum-commuting} if there exists a C*-algebra $\mathfrak{A}$, projection valued measures $\{E_{x,i}\}, \{F_{y,j}\} \subseteq \mathfrak{A}$ satisfying $E_{x,i} F_{y,j} = F_{y,j} E_{x,i}$ and a state $\phi$ on $\mathfrak{A}$ such that $p(i,j|x,y) = \phi(E_{x,i} F_{y,j})$. The set of all quantum-commuting correlations is denoted by $C_{qc}(n,m)$.
    \item A correlation $p \in C(n,m)$ is called a \textbf{quantum} correlation if $p$ is a quantum-commuting correlation arising from a finite dimensional C*-algebra $\mathfrak{A}$. The set of quantum correlations is denoted by $C_q(n,m)$.
    \item A correlation $p \in C(n,m)$ is called a  \textbf{quantum-approximate} correlation if it is a limit point of the quantum correlations in $\mathbb{R}^{n^2 m^2}$. We let $C_{qa}(n,m)$ denote the set of all quantum approximate correlations.
    \item A correlation $p \in C(n,m)$ is called a \textbf{local} correlation if $p$ is a quantum-commuting correlation arising from a commutative C*-algebra $\mathfrak{A}$. The set of local quantum correlations is denoted by $C_{loc}(n,m)$. \end{enumerate} \end{definition}

Morally, a correlation is a joint probability distribution produced by two independent parties, usually labeled Alice and Bob, who each have $n$ probabilistic experiments they can perform, each experiment having $m$ possible outcomes. The quantity $p(i,j|x,y)$ indicates the probability that Alice obtains outcome $i$ and Bob obtains outcome $j$ given that Alice performed experiment $x$ and Bob performed experiment $y$. The different types of correlation sets arise by imposing various restrictions on the technology Alice and Bob can employ to perform their joint experimentation.

It is well-known that the correlation sets satisfy \[ C_{loc}(n,m) \subseteq C_{q}(n,m) \subseteq C_{qa}(n,m) \subseteq C_{qc}(n,m) \subseteq C_{ns}(n,m) \subseteq C(n,m) \] and that each of these sets is convex. Moreover each of these sets, with the exception of $C_q(n,m)$, is known to be closed. In fact, $C_{loc}(n,m), C_{ns}(n,m)$ and $C(n,m)$ are easily seen to be polytopes. By contrast, the sets $C_q(n,m), C_{qa}(n,m)$, and $C_{qc}(n,m)$ are not polytopes for most choices of $n$ and $m$. We can now state the strong and weak Tsirelson problems.

\begin{problem}[Strong Tsirelson] Is $C_q(n,m) = C_{qc}(n,m)$? \end{problem}

Since $C_{qc}(n,m)$ is closed for all $n$ and $m$, the answer is negative whenever $C_q(n,m)$ fails to be a closed set, which is known to occur whenever $n \geq 5$ and $m \geq 2$ \cite{DykemaPaulsenPrakash2019} or whenever $n \geq 4$ and $m \geq 3$ \cite{CqNeqCqs}. The first example of $n$ and $m$ for which $C_q(n,m)$ is non-closed is due to Slofstra \cite{Slofstra1}. The strong Tsirelson problem remains unsolved for $n=3$ and all $m$ or $n=4$ and $m=2$. The only setting in which it is known to be true is the setting $n=m=2$.

\begin{problem}[Weak Tsirelson] Is $C_{qa}(n,m) = C_{qc}(n,m)$? \end{problem}

Clearly an affirmative answer to the strong Tsirelson problem implies an affirmative answer to the weak Tsirelson problem. To our knowledge, the weak Tsirelson problem has only been answered in the case when $n=m=2$ (in the affirmative), and is open in all other settings. The difficulty of this problem is explained by the next theorem.

\begin{theorem}[Ozawa \cite{MR3067294}] Connes' embedding problem has an affirmative answer if and only if the weak Tsirelson problem has an affirmative answer for all $n,m \in \mathbb{N}$. \end{theorem}

We now turn our attention to the subset of synchronous correlations which we define now.

\begin{definition} A correlation $p \in C(n,m)$ is called \textbf{synchronous} if $p(i,j|x,x) = 0$ whenever $i \neq j$. For each $r \in \{loc, q, qa, qc, ns\}$ we let $C_r^s(n,m)$ denote the set of all synchronous correlations in $C_r(n,m)$. \end{definition}

A state $\tau$ on a C*-algebra $\mathfrak{A}$ is called tracial if $\tau(xy)=\tau(yx)$ for all $x,y \in \mathfrak{A}$. A useful characterization of the synchronous quantum correlations in terms of tracial states was discovered by Paulsen-et.al. in \cite{MR3460238}. We state this result here.

\begin{theorem}[Paulsen et. al.] \label{PaulsenWinter} A correlation $p \in C_{qc}^s(n,m)$ if and only if there exists a C*-algebra $\mathfrak{A}$, projection-valued measures $\{E_{x,i}\} \subseteq \mathfrak{A}$ and a tracial state $\tau$ on $\mathfrak{A}$ such that $p(i,j|x,y) = \tau(E_{x,i} E_{y,j})$. If the C*-algebra is finite-dimensional, then the corresponding correlation $p$ is in $C_q^s(n,m)$. If the C*-algebra is commutative, then the corresponding correlation $p$ is in $C_{loc}^s(n,m)$. \end{theorem}

A state $\phi$ on a C*-algebra is called faithful if $\phi(x^*x) > 0$ whenever $x \neq 0$. A simple application of the GNS construction shows that we may always take the state $\tau$ in Theorem \ref{PaulsenWinter} to be faithful, a fact that we will exploit later.

As in the non-synchronous case, it is well-known that \[ C_{loc}^s(n,m) \subseteq C_q^s(n,m) \subseteq C_{qa}^s(n,m) \subseteq C_{qc}^s(n,m) \subseteq C_{ns}^s(n,m) \subseteq C^s(n,m) \] and that these sets are convex. With execption of $C_q^s(n,m)$ all of these sets are closed, and the sets $C_{loc}^s(n,m)$, $C_{ns}^s(n,m)$ and $C^s(n,m)$ are polytopes. The sets $C_q^s(n,m)$, $C_{qa}^s(n,m)$, and $C_{qc}^s(n,m)$ are not polytopes for most $n$ and $m$. Moreover we have a reformulation of Connes' problem in this setting as well.

\begin{theorem}[Dykema-Paulsen \cite{DykemaPaulsen2016}] Connes' embedding problem has an affirmative answer if and only if the closure of $C_q^s(n,m)$ coincides with $C_{qc}^s(n,m)$ for all $n,m \in \mathbb{N}$. \end{theorem}

In \cite{DykemaPaulsen2016} Dykema-Paulsen ask whether or not the closure of $C_q^s(n,m)$ coincides with $C_{qa}^s(n,m)$. Their question was answered affirmatively by Kim-Paulsen-Schafhauser in \cite{KimPaulsenSchafhauser2018}. The proof uncovered another interesting connection between Connes' problem and synchronous correlation sets, which we summarize in the next theorem.

\begin{theorem}[Kim-Paulsen-Schafhauser \cite{KimPaulsenSchafhauser2018}] The set $C_{qa}^s(n,m)$ is equal to the closure of $C_q^s(n,m)$ for each $n,m \in \mathbb{N}$. Moreover, a correlation $p \in C_{qa}^s(n,m)$ if and only if there exist projection-valued measures $\{ E_{x,i}\}_{i \in [m]} \subseteq \mathcal{R}^\omega$ for each $x \in [n]$ such that $p(i,j|x,y) = \tau(E_{x,i} E_{y,j})$, where $\tau$ is the unique trace on $\mathcal{R}^\omega$. \end{theorem}

We conclude with a discussion about the synchronous correlation sets in the two-outcome senario. Assume $p \in C_{qc}^s(n,2)$. Then for each $x,y \in [n]$ the synchronous and non-signaling conditions imply that \begin{equation} \label{eqn1} p(i,j|x,x) = \begin{pmatrix} w_{x,x} & 0 \\ 0 & 1-w_{x,x} \end{pmatrix} \end{equation} and \begin{equation} \label{eqn2} p(i,j|x,y) = \begin{pmatrix} w_{x,y} & w_{x,x} - w_{x,y} \\ w_{y,y} - w_{x,y} & 1 + w_{x,y} - w_{x,x} - w_{y,y} \end{pmatrix} \end{equation} where $w_{x,y} = p(0,0|x,y)$. Since $p \in C_{qc}^s(n,2)$ there exists a $C^*$-algebra $\mathfrak{A}$, projection valued measures $\{E_{x,i}\} \subseteq \mathfrak{A}$ and a tracial state $\tau$ on $\mathfrak{A}$ such that $p(i,j|x,y) = \tau(E_{x,i} E_{y,j})$. Therefore $w_{x,y} = w_{y,x}$. It follows that the data of $p$ is contained in the symmetric matrix $(w_{x,y})$. Following the notation of \cite{Musat2019}, we define $D_{qc}(n)$ to be the set of symmetric matrices with entries of the form $d_{i,j} = \tau(P_i P_j)$ where $\{P_i\} \subseteq \mathfrak{A}$ are projections and $\tau$ is a tracial state on $\mathfrak{A}$. Equivalently $D_{qc}(n)$ is the image of the restriction map $\pi: C_{qc}(n,2) \rightarrow \mathbb{R}^{n^2}$ given by $\pi(p)(x,y) = p(0,0|x,y)$. Similarly we define $D_{qa}(n)$, $D_{q}(n)$ and $D_{loc}(n)$. By Equations (\ref{eqn1}) and (\ref{eqn2}) above it is evident that $\pi$ is an affine isomorphism between $C_r(n,2)$ and $D_r(n)$. We summarize this discussion with the following Proposition.

\begin{proposition} \label{bijectionProp2} The restriction map $\pi: C_r(n,2) \rightarrow \mathbb{R}^{n^2}$ given by $\pi(p)(x,y) = p(0,0|x,y)$ is an affine bijection onto its range $D_r(n)$ for each $r \in \{loc, q, qa, qc, ns\}$. \end{proposition}

\section{Connes' embedding problem and two-outcome correlation sets}

In this section we wish to establish an equivalence between Connes' problem and the weak Tsirelson problem in the setting of two-outcome synchronous correlation sets. Our strategy will be to identify certain convex faces of the two-outcome synchronous correlation sets which are affine images of the $m$-outcome synchronous correlation sets. We will make use of the following Lemma.

\begin{lemma} \label{PVMLemma} Let $\mathfrak{A}$ be a $C^*$-algebra with a faithful tracial state $\tau: \mathfrak{A} \rightarrow \mathbb{C}$. Suppose that $\{P_i\}_{i=1}^m$ are projections in $\mathfrak{A}$ satisfying $\tau(P_i P_j) = 0$ when $i \neq j$. Then $P_i P_j = 0$ whenever $i \neq j$. If in addition $\sum_{i,j} \tau(P_i P_j) = 1$, then $\{P_i\}_{i=1}^m$ is a projection valued measure. \end{lemma}

\begin{proof} Since the $P_i$'s are projections and $\tau$ is a trace, we have \[ \tau(P_i P_j) = \tau(P_j P_i P_i P_j) = \tau((P_i P_j)^* (P_i P_j)). \] Since $\tau$ is faithful, it follows that $P_i P_j = 0$ whenever $i \neq j$. Consequently $P := \sum_i P_i$ is a projection. If \[ 1 = \sum_{i,j} \tau(P_i P_j) = \tau(P). \] then $\tau(I - P) = 0$. Since $P \leq I$, $I - P$ is also a projection and therefore $I - P = 0$ since $\tau$ is faithful. \end{proof}

\begin{definition} \label{defn} Let $n,m \in \mathbb{N}$, and fix some bijection $[n] \times [m] \rightarrow [nm]$ so that elements of $[nm]$ are indexed as ordered pairs $(x,i)$ with $x \in [n]$ and $i \in [m]$. Define a map $\pi: C(nm,2) \rightarrow \mathbb{R}^{n^2m^2}$ via $\pi(p)(i,j|x,y) := p(0,0|(x,i),(y,j))$. Then we define \[ F_r^s(n,m) := \{ p \in C_r^s(nm,2) : \pi(p) \in C_{ns}^s(n,m) \}. \] \end{definition}

We remark that our definition of $F_r^s(n,m)$ depends on the labeling of $[nm]$. However it should be transparent that all such choices lead to affinely isomorphic subsets of $C_r^s(nm,2)$ so that this technicality is unimportant. Moreover, the map $\pi$ in this definition is the same $\pi$ from Proposition \ref{bijectionProp2}, making $F_r^s(n,m)$ the preimage of an affine slice of $D_r(nm)$. The following results establish the basic properties of $F_r^s(n,m)$ and the existence of an affine isomorphism with $C_r^s(n,m)$ for $r \in \{loc, q, qa, qc\}$.

\begin{proposition} \label{BijectionProp} For each $n,m \in \mathbb{N}$, $p \in F_r^s(n,m)$ if and only if $\pi(p) \in C_r^s(n,m)$ for each $r \in \{loc,q,qa,qc, ns\}$. \end{proposition}

\begin{proof} The case when $r=ns$ is obvious by Definition \ref{defn} (that $\pi^{-1}$ is a well-defined affine map follows from Proposition \ref{bijectionProp2}). We consider the other cases. Let $\mathfrak{A}$ be a $C^*$-algebra with a faithful tracial state $\tau$ and let $\{E_{a,(x,i)} \}_{a=0,1} \subseteq \mathfrak{A}$ be projection-valued measures for each $(x,i) \in [n] \times [m]$. Set $p(a,b|(x,i),(y,j)) = \tau(E_{a,(x,i)} E_{b,(y,j)})$. If $p \in F_{qc}^s(n,m)$ then $\pi(p) \in C_{ns}^s(n,m)$. Consequently $\tau(E_{0,(x,i)} E_{0,(x,j)}) = 0$ whenever $i \neq j$ and $\sum_{i,j} \tau(E_{0,(x,i)} E_{0,(x,j)}) = 1$. By Lemma \ref{PVMLemma} it follows that $\{E_{0,(x,i)}\}_{i=1}^m$ is a projection valued measure for each $x \in [n]$. Therefore $\pi(p) \in C_{qc}^s(n,m)$. Taking $\mathfrak{A}$ to be $\mathcal{R}^\omega$, finite-dimensional, or commutative and repeating the above argument we see that $p \in F_{r}^s(n,m)$ implies that $\pi(p) \in C_{r}^s(n,m)$ for $r \in \{ loc, q, qa\}$.

On the other hand, suppose $q \in C_{qc}^s(n,m)$ with $q(i,j|x,y) = \tau(F_{x,i} F_{y,j})$. Then we can build $p \in C_{qc}^s(nm,2)$ by setting $E_{0,(x,i)} = F_{x,i}$ and $E_{1,(x,i)} = I - F_{x,i}$ and letting $p(a,b|(x,i),(y,j)) = \tau(E_{a,(x,i)} E_{b,(y,j)})$. By the non-signaling conditions, we see that $q = \pi(p)$. Taking $\mathfrak{A}$ to be $\mathcal{R}^\omega$, finite-dimensional or commutative we obtain the other results. \end{proof}

\begin{proposition} \label{FaceProp} For each $n,m \in \mathbb{N}$, $F_r^s(n,m)$ is a relatively closed face of $C_r^s(nm,2)$ for each $r \in \{loc,q,qa,qc\}$. \end{proposition}

\begin{proof} It is clear that $F_r^s(n,m)$ is relatively closed since it is the intersection of $C_r^s(nm,2)$ with the affine space defined by the non-signaling conditions. Suppose that $p \in F_r^s(n,m)$, $q,s \in C_r^s(nm,2)$ and $p = \lambda q + \mu s$ with $\lambda + \mu = 1$ and $\lambda, \mu \geq 0$. Since $\pi(p) \in C_{ns}^s(n,m)$, we have \[ \lambda q(0,0|(x,i),(x,j)) + \mu s(0,0|(x,i),(x,j)) = 0 \]  whenever $i \neq j$, and hence $q(0,0|(x,i),(x,j)) = s(0,0|(x,i),(x,j)) = 0$. 

Now if $q,s \in C_{qc}^s(nm,2)$, then $q(a,b|(x,i),(y,j)) = \tau(E_{a,(x,i)} E_{b,(y,j)})$ and $s(a,b|(x,i),(y,j)) = \rho(F_{a,(x,i)} F_{b,(y,j)})$ as in Theorem \ref{PaulsenWinter}. However, by Lemma \ref{PVMLemma}, we have $E_{0,(x,i)} E_{0,(x,j)} = 0$ and $F_{0,(x,i)} F_{0,(x,j)} = 0$ whenever $i \neq j$. It follows that $\sum_{i,j} E_{0,(x,i)} E_{0,(x,j)} = \sum_i E_{0,(x,i)} \leq I$ and hence $\sum_{i,j} q(0,0|(x,i),(x,i)) \leq 1$. Likewise, $\sum_{i,j} s(0,0|(x,i),(x,j)) \leq 1$. Consequently \[ \lambda (\sum_{i,j} q(0,0|(x,i),(x,j))) + \mu (\sum_{i,j} s(0,0|(x,i),(x,j))) = 1 \] implies that \[ \sum_{i,j} q(0,0|(x,i),(x,j)) = \sum_{i,j} s(0,0|(x,i),(x,j))=1.\] Therefore $\{E_{0,(x,i)}\}_{i=1}^m$ and $\{F_{0,(y,j)} \}_{j=1}^m$ are projection valued measures by Lemma \ref{PVMLemma}. It follows that $q,s \in F_{qc}^s(n,m)$. The remaining statements follow by assuming the correlations are realized in $\mathcal{R}^\omega$, some finite-dimensional C*-algebra, or some commutative C*-algebra and repeating the above arguments. \end{proof}

Before stating the main theorem of this section we should include a remark about the above proof. The reader may notice that if $F_{ns}^s(n,m)$ were a face of $C_{ns}^s(nm,2)$ then it would follow that $F_{r}^s(n,m)$ is a face of $C_{r}^s(nm,2)$ for each $r \in \{loc, q, qa, qc \}$ by Proposition \ref{BijectionProp}. This would be a better result and a more efficient proof. However the proposition is false for $r=ns$. Indeed, define the following matrices for each $x, y \in [2]$ and $x \neq y$: \[ p(i,j|x,x) = \begin{pmatrix} .5 & 0 \\ 0 & .5 \end{pmatrix}, \quad p(i,j|x,y) = \begin{pmatrix} .25 & .25 \\ .25 & .25 \end{pmatrix}, \] \[ q(i,j|x,x) = \begin{pmatrix} .5 & 0 \\ 0 & .5 \end{pmatrix}, \quad q(i,j|x,y) = \begin{pmatrix} .5 & .5 \\ .5 & .5 \end{pmatrix}, \] \[ s(i,j|x,x) = \begin{pmatrix} .5 & 0 \\ 0 & .5 \end{pmatrix}, \quad s(i,j|x,y) = \begin{pmatrix} 0 & 0 \\ 0 & 0 \end{pmatrix}. \] Then $p$ defines a correlation in $C_{ns}^s(2,2)$ while $q,s$ fail the non-signaling conditions and thus are not correlations. However $\pi^{-1}(p) \in F_{ns}^s(2,2)$, $\pi^{-1}(q), \pi^{-1}(s) \in C_{ns}^s(4,2)$ and \[ \pi^{-1}(p) = \frac{1}{2} \pi^{-1}(q) + \frac{1}{2} \pi^{-1}(r). \]

We conclude with the main theorem of this section.

\begin{theorem} \label{ConnesThm1} Connes' embedding problem has an affirmative answer if and only if $C_{qa}^s(n,2) = C_{qc}^s(n,2)$ for every $n \in \mathbb{N}$. \end{theorem}

\begin{proof} By \cite{DykemaPaulsen2016} and \cite{KimPaulsenSchafhauser2018} we know that Connes' embedding conjecture is equivalent to $C_{qa}^s(n,m) = C_{qc}^s(n,m)$ for all $n$ and $m$. Thus it suffices to demonstrate equivalence of $C_{qa}^s(n,2) = C_{qc}^s(n,2)$ for all $n$ with $C_{qa}^s(n,m) = C_{qc}^s(n,m)$ for all $n$ and $m$.

Clearly if $C_{qa}^s(n,m) = C_{qc}^s(n,m)$ for all $n$ and $m$ then $C_{qa}^s(n,2) = C_{qc}^s(n,2)$ for all $n$. Now suppose that $C_{qa}^s(n,m) \neq C_{qc}^s(n,m)$ for some $n$ and $m$. Then by Proposition \ref{BijectionProp} we have $F_{qa}^s(n,m) \neq F_{qc}^s(n,m)$. However $F_{qa}^s(n,m)$ and $F_{qc}^s(n,m)$ are closed faces of $C_{qa}^s(nm,2)$ and $C_{qc}^s(nm,2)$, respectively, by Proposition \ref{FaceProp}. Therefore $C_{qa}^s(nm,2) \neq C_{qc}^s(nm,2)$. \end{proof}

As a corollary we obtain an unpublished result of Ozawa.

\begin{corollary}[Ozawa \cite{OzawaUnpublished}] \label{OzawaCor} Connes' embedding problem has an affirmative answer if and only if $C_{qa}(n,2) = C_{qc}(n,2)$ for every $n \in \mathbb{N}$. \end{corollary}

\begin{proof} If Connes' conjecture is true then $C_{qa}(n,2) = C_{qc}(n,2)$ for every $n$. On the other hand, if Connes' conjecture is false, then by Theorem \ref{ConnesThm1} we have $C_{qa}^s(n,2) \neq C_{qc}^s(n,2)$ for some $n$. This implies that $C_{qa}(n,2) \neq C_{qc}(n,2)$ since the synchronous correlations constitute an affine slice of the non-synchronous correlations. \end{proof}

\begin{remark} \emph{After a first preprint of this paper was made available, the author was contacted by Samuel Harris who pointed out that the main result of this section (Theorem \ref{ConnesThm1}) also appears in his thesis (\cite{HarrisThesis}, Theorem 1.9.5). The proofs are similar. We have left this section unchanged since it includes several additional details about the embedding of $C_r^s(n,m)$ into $C_r^s(nm,2)$ and because these results are used in and motivate the results of the next section.} \end{remark}

\section{Extreme points of two-outcome quantum correlations}

In this section we will consider the sets $D_{r}(n)$ for $r \in \{loc, q, qa, qc \}$ which, by Proposition \ref{bijectionProp2}, is affinely isomorphic to $C_r(n,2)$. We will describe the geometry of $D_r(n)$ in terms of a family of subsets indexed by vectors in $[0,1]^n$, which we define now.

\begin{definition} For each $\vec{y} \in [0,1]^n$ and $r \in \{loc, q, qa, qc \}$, we define the $\boldsymbol{\vec{y}}$\textbf{-slice} of $D_r(n)$, denoted by $S_{\vec{y}}[D_r(n)]$, to be the set of all vectors \[ (w_{i,j})_{i < j < n} \subseteq \mathbb{R}^{\frac{(n-1)(n-2)}{2}} \] with the following property: there exists $p \in D_r(n)$ such that $p_{i,j} = w_{i,j}$ for each $0 \leq i < j < n$ and $p_{i,i} = y_i$ for each $i \in [n]$. \end{definition}

\noindent Essentially the $\vec{y}$-slice of $D_r(n)$ is the projection of the subset of $D_r(n)$ with diagonal entries given by $\vec{y}$ onto its upper diagonal entries. Since $D_r(n)$ is affinely isomorphic to $C_r^s(n,2)$ and since elements of $D_r(n)$ are symmetric matrices, we see that the geometry of $C_r^s(n,2)$ is determined by the $\vec{y}$-slices of $D_r(n)$. We remark that since $\tau(P_i P_j) \leq \min(\tau(P_i), \tau(P_j))$ for every choice of $P_i$ and $P_j$, we see that $S_{\vec{r}}[ D_{qc}(n)]$ is a subset of the $\frac{1}{2}(n-1)(n-2)$-dimensional rectangular prism whose $(i,j)$ coordinate lies in the interval $[0, \min(y_i, y_j)]$. 

We will analyze the geometry of the $\vec{y}$-slices of $D_r(n)$ by considering their projections onto lines given by $\Span( \vec{x})$ where $\vec{x} = (x_{i,j})_{i < j < n} \in \mathbb{R}^{\frac{(n-1)(n-2)}{2}}$ is any non-zero vector. To this end we define for each $r \in \{loc, q, qc\}$, vector $\vec{x} = (x_{i,j})_{i < j < n}$ and vector $\vec{y} \in [0,1]^n$ the quantities \[ u_r(\vec{y}, \vec{x}) = \sup \{ \sum_{i < j < n} x_{i,j} w_{i,j} : w \in S_{\vec{y}}[D_r(n)] \} \] and \[ l_r(\vec{y}, \vec{x}) = \inf \{ \sum_{i < j < n} x_{i,j} w_{i,j} : w \in S_{\vec{y}}[D_r(n)] \}. \] When $\vec{x}$ is a unit vector the set $[l_{r}(\vec{y}, \vec{x}), u_{r}(\vec{y}, \vec{x})] \vec{x}$ is exactly the closure of the projection of $S_{\vec{r}}[ D_{qc}(n)]$ onto the ray $\Span{(\vec{x})}$. Otherwise it is a rescaling of this projection.

The functional $l_r(\vec{y}, \vec{x})$ directly generalizes the graph functional $f_r(t)$ defined by Dykema-Paulsen-Prakash in \cite{DykemaPaulsenPrakash2019}. In fact, settting $\vec{y}$ equal to the vector with constant entries of $t$ and $\vec{x}$ equal to the constant vector with entries of $1$ we get $f_r(t) = 2 l_r(\vec{y}, \vec{x}) + nt$. The next theorem demonstrates why we are interested in the quantities $l_{r}(\vec{y}, \vec{x})$ and $u_{r}(\vec{y}, \vec{x})$.

\begin{theorem} Connes' embedding problem has an affirmative solution if and only if for every $n \in \mathbb{N}$, $\vec{y} \in [0,1]^n$ and $\vec{x} \in \mathbb{R}^{\frac{(n-1)(n-2)}{2}}$ we have \[ u_q(\vec{y}, \vec{x}) = u_{qc}(\vec{y}, \vec{x}). \] \end{theorem}

\begin{proof} If Connes's conjecture is true then $D_{qc}(n) = \overline{D_{q}(n)}$ for every $n$ and consequently \[ u_q(\vec{y}, \vec{x}) = u_{qc}(\vec{y}, \vec{x}) \] for every $\vec{y}$ and $\vec{x}$. Now suppose Connes's conjecture is false. Then there exists some $n$ such that $D_{qc}(n) \neq D_{qa}(n)$ by Theorem \ref{ConnesThm1}. Thus $S_{\vec{y}}[D_{qc}(n)] \neq S_{\vec{y}}[D_{qa}(n)]$ for some $\vec{y}$. Suppose that $p \in S_{\vec{y}}[D_{qc}(n)] \setminus S_{\vec{y}}[D_{qa}(n)]$. By the Hahn-Banach theorem there exists a hyperplane that strictly separates $p$ from $S_{\vec{y}}[D_{qa}(n)]$. Taking $\vec{x}$ to be a vector normal to this hyperplane, we get \[ [l_{qc}(\vec{y},\vec{x}), u_{qc}(\vec{y},\vec{x})] \neq [l_{q}(\vec{y},\vec{x}), u_{q}(\vec{y},\vec{x})]. \] Finally notice that that $l_r(\vec{y}, \vec{x}) = - u_r(\vec{y}, -\vec{x})$. \end{proof}

Since $D_{qc}(n)$ is a compact set, there exists a $C^*$-algebra $\mathfrak{A}$, projections $\{P_i\} \subseteq \mathfrak{A}$ and a tracial state $\tau$ on $\mathfrak{A}$ such that \[ u_{r}(\vec{y}, \vec{x}) = \sum_{i < j < n} x_{i,j} \tau(P_i P_j). \] We will show that the projections $\{P_i\}$ achieving this value satisfy certain relations. This follows from the next proposition.

\begin{proposition}[Dykema-Paulsen-Prakash, \cite{DeltaGame}] \label{DPPproposition} Let $\tau$ be a trace on a $C^*$-algebra $\mathfrak{A}$ and let $A,B \in \mathfrak{A}$ be hermitian with $[A,B] \neq 0$. Then there exists a hermitian $H \in \mathfrak{A}$ such that the function $f(t) = \tau(A e^{iHt} B e^{-iHt})$ satisfies $f'(0) > 0$. \end{proposition}

\begin{proposition} \label{CommutationRelations} Suppose that $ u_r(\vec{y}, \vec{x}) = \sum_{i < j < n} x_{i,j} \tau(P_i P_j)$. Then \[ [P_i, \sum_{i \neq j} x_{i,j} P_j] = 0. \] \end{proposition}

\begin{proof} Set $A = \sum_{i \neq j} x_{i,j} P_j$. Since $\{P_i\}$ satisfy $ u_r(\vec{y}, \vec{x}) = \sum_{i < j < n} x_{i,j} \tau(P_i P_j) = \tau(P_i A)$, it follows from Proposition \ref{DPPproposition} that $[P_i, A]=0$. Otherwise we could choose a hermitian operator $H$ and a small $t > 0$ such that \[ u_r(\vec{y}, \vec{x}) < \sum_{i < j < n} x_{i,j} \tau(Q_i Q_j) \] where $Q_i = P_i$ and $Q_j = e^{iHt}P_je^{-iHt}$ for each $j \neq i$. Since $\tau(Q_i) = \tau(P_i) = y_i$ this contradicts the definition of $u_r(\vec{y}, \vec{x})$. \end{proof}

In light of the preceding proposition, we define a family of C*-algebras in which one can produce the extreme points of $D_r(n)$.

\begin{definition} For each $n \in \mathbb{R}$ and $\vec{x} \in \mathbb{R}^{\frac{(n-1)(n-2)}{2}}$, we define the C*-algebra $\mathfrak{A}_{\vec{x}}$ to be the universal C*-algebra generated by projections $P_0, P_1, \dots, P_{n-1}$ satisfying the relations
\[ [P_i, \sum_{i \neq j} x_{i,j} P_j] = 0 \] for each $i \in [n]$.
\end{definition}

\begin{theorem} \label{ExtremePointThm} Let $r \in \{loc, q, qc\}$ and suppose that $p$ is an extreme point of $D_r(n)$. Then there exists a vector $\vec{x} \in \mathbb{R}^{\frac{(n-1)(n-2)}{2}}$ and a trace $\tau$ on  $\mathfrak{A}_{\vec{x}}$ such that $p(i,j) = \tau(P_i P_j)$ and $\tau(P_i) = y_i$. \end{theorem}

\begin{proof} Since $p \in D_{qc}(n)$ there exists a C*-algebra $\mathfrak{B}$, a trace $\rho$ on $\mathfrak{B}$ and projections $Q_0, \dots, Q_{n-1}$ such that $p(i,j)=\rho(Q_i Q_j)$. Since $p$ is an extreme point of $D_r(n)$, its upper triangular entries (which we also denote $p$) must produce an extreme point for $S_{\vec{y}}[D_r(n)]$ where $y_i = p(i,i)$. By the Hahn-Banach Theorem, there exists a hyperplane in $\mathbb{R}^{\frac{(n-1)(n-2)}{2}}$ containing $p$ for which $S_{\vec{y}}[D_r(n)]$ is contained in one of its half-spaces. Taking $\vec{x}$ to be a vector normal to this hyperplane we get $u_r(\vec{y}, \vec{x}) = \sum_{i < j < n} x_{i,j} \tau(P_i P_j)$. By Proposition \ref{CommutationRelations} the projections $\{Q_i\}$ satisfy $[Q_i, \sum_{i \neq j} x_{i,j} Q_j] = 0$ for each $i \in [n]$. By the universality of $\mathfrak{A}_{\vec{x}}$, there exists $*$-homomorphism $\pi: \mathfrak{A}_{\vec{x}} \rightarrow \mathfrak{B}$ satisfying $\pi(P_i) = \pi(Q_i)$ for each $i \in [n]$. Setting $\tau = \rho \circ \pi$ yields a trace satisfying the conditions of the theorem. \end{proof}

We consider the special case of $D_{qc}(3)$. In this case, the relations of Proposition \ref{CommutationRelations} become \[ [A,aB + bC] = [B, aA + cC] = [C, bA + cB] = 0 \] where $A := P_0, B := P_1, C:= P_2, a:=x_{0,1}, b:=x_{0,2}$, and $c:=x_{1,2}$. Note that $[A,aB+bC] = [B,aA+cC] = 0$ implies that $[C,bA+cB]=0$ since \[ -[C,bA+cB] = [A,aB+bC] + [B,aA+cC]. \] Finally, we note that to consider an arbitrary ray paralell to $(a,b,c)$ with each of $a,b,c$ non-zero we may assume without loss of generality that $c=1$ by rescaling.

\begin{theorem} \label{universal3} Let $\mathfrak{A}$ be the universal $C^*$-algebra generated by projections $A,B,C$ satisfying the relations \[ [A,aB+bC] = [B,aA+C] = 0 \] for non-zero $a,b \in \mathbb{R}$. Then $\mathfrak{A} \cong \mathbb{C}^8 \oplus \mathbb{M}_2$ with \begin{eqnarray} A & = & 1 \oplus 1 \oplus 1 \oplus 1 \oplus 0 \oplus 0 \oplus 0 \oplus 0 \oplus \begin{pmatrix} 1 & 0 \\ 0 & 0 \end{pmatrix} \nonumber \\ B & = & 1 \oplus 1 \oplus 0 \oplus 0 \oplus 1 \oplus 1 \oplus 0 \oplus 0 \oplus \begin{pmatrix} t & \sqrt{t(1-t)} \\ \sqrt{t(1-t)} & 1-t \end{pmatrix} \nonumber \\ C & = & 1 \oplus 0 \oplus 1 \oplus 0 \oplus 1 \oplus 0 \oplus 1 \oplus 0 \oplus \begin{pmatrix} z & -\frac{a}{b}\sqrt{t(1-t)} \\ -\frac{a}{b}\sqrt{t(1-t)} & 1-z \end{pmatrix} \nonumber \end{eqnarray} where $z = \frac{1}{2} \pm \frac{1}{2} \sqrt{1 - \frac{4a^2}{b^2}t(1-t)}$ and $t = \frac{b^2 + 2a^2b - a^2b^2 - a^2}{4a^2b}$. \end{theorem}

We remark that there are two choices for $z$ in the above theorem. As we will show in the proof, exactly one of the two choices for $z$ leads to operators satisfying the conditions of the theorem.

\begin{proof} We first show that each irreducible representation of $\mathfrak{A}$ has dimension no more that four. Consider the operator \[ H = (aB + bC)^2 - (a+b)(aB + bC). \] By assumption $H$ commutes with $A$, and by an easy computation $H$ commutes with $B$ and $C$ as well. It follows that any irreducible representation of $\mathfrak{A}$ must map $H$ to a scalar multiple of the identity. It follows that under any such representation $\pi$ we have $\pi(CB) \in \Span(I,\pi(B),\pi(C),\pi(BC))$. Similar calculations show that $\pi(BA) \in \Span(I,\pi(A),\pi(B),\pi(AB))$ and $\pi(CA) \in \Span(I,\pi(A),\pi(C),\pi(AC))$. Consequently \[ \pi(\mathfrak{A}) \subseteq \Span(I,\pi(A),\pi(B),\pi(C),\pi(AB),\pi(AC),\pi(BC),\pi(ABC)) \] and hence $\dim(\pi(\mathfrak{A})) \leq 8$. Since $\pi(\mathfrak{A})$ must be a matrix algebra it must be isomorphic to $\mathbb{C}$ or $\mathbb{M}_2$. 

Now there are precisely eight one-dimensional representations of $\mathfrak{A}$ given by the eight possible ways to map the projections $A,B,$ and $C$ to 0 or 1. It remains to consider the four-dimensional representations of $\mathfrak{A}$. We will show that up to unitary equivalence there is only one four-dimensional representation of $\mathfrak{A}$. Indeed, assume we have some irreducible representation $\pi: \mathfrak{A} \rightarrow \mathbb{M}_2$. Then we may assume without loss of generality that $A$ and $B$ do not commute and hence \[ \pi: \quad A \mapsto \begin{pmatrix} 1 & 0 \\ 0 & 0 \end{pmatrix}, \quad B \mapsto \begin{pmatrix} t & \sqrt{t(1-t)} \\ \sqrt{t(1-t)} & 1-t \end{pmatrix} \] for some $t \in (0,1)$ with respect to some choice of basis for $\mathbb{C}^2$. Since $aB + bC$ commutes with $A$, we must have \[ \pi: C \mapsto \begin{pmatrix} z & -\frac{a}{b}\sqrt{t(1-t)} \\ -\frac{a}{b}\sqrt{t(1-t)} & 1-z \end{pmatrix} \] for some choice of $z \in (0,1)$. Using the relation $C^2=C$ we see that \[ z = \frac{1}{2} \pm \frac{1}{2} \sqrt{1 - \frac{4a^2}{b^2}t(1-t)}. \] Finally a tedious calculation shows that the relation $[B, aA + C]=0$ holds only if $t = \frac{b^2 + 2a^2b - a^2b^2 - a^2}{4a^2b}$. The choice of $z$ then depends on the sign of $a$ and the sign of $a^2(b^2-1) + b^2$. We leave these details to the interested reader. It follows that there is only one four-dimensional representation of $\mathfrak{A}$ up to unitary equivalence and hence the isomorphism described in the statement of the theorem is faithful. \end{proof}

\begin{theorem} The sets $C_q^s(3,2)$ and $C_{qc}^s(3,2)$ coincide. Consequently $C_q^s(3,2)$ is closed. \end{theorem}

\begin{proof} Let $w$ be an extreme point of $D_{qc}(3)$. Then by Theorem \ref{ExtremePointThm} there exists a vector $(a,b,c) \in \mathbb{R}^3$, projections $A,B,C \subset \mathfrak{A}_{(a,b,c)}$, and a trace $\tau$ on $\mathfrak{A}_{(a,b,c)}$ such that $w_{0,1} = \tau(AB), w_{0,2} = \tau(AC)$, and $w_{1,2} = \tau(BC)$. If $a,b,c$ are all non-zero, we see by Theorem \ref{universal3} that $\mathfrak{A}$ is finite dimensional. For the other cases, let us first assume that $a = 0$. Then we have \[ u_{qc}(\vec{y}, \vec{x}) = b \tau(AC) + c \tau(BC) \] with $(\tau(A),\tau(B),\tau(C)) = \vec{y}$. Thus to maximize $u_{qc}(\vec{y}, \vec{x})$ it suffices to maximize $\tau(AC)$ if $b > 0$ or minimize $\tau(AC)$ if $b < 0$, and simultaneously to maximize $\tau(BC)$ if $c > 0$ or minimize $\tau(BC)$ if $c < 0$. This can be achieved in a commutative C*-algebra. Indeed, for any trace $\tau$ and projections $P$ and $Q$ we have \[ \tau(PQ) \in [ \max(0, \tau(P)+\tau(Q)-1), \min(\tau(P),\tau(Q))]. \] It is easy to check that we can choose projections $A,B,C \in L^\infty[0,1]$, the commutative C*-algebra of $L^\infty$ functions on the unit interval [0,1], such that $\tau(AC)$ and $\tau(BC)$ take either the maximum or minimum values in the above interval as needed, where $\tau$ is the Lebesgue integral. For example, to maximize $\tau(AC)$ and minimize $\tau(BC)$, we could identify $A, B$, and $C$ with the functions \begin{eqnarray} f_A(t) & = & \begin{cases} 1 & t \in [0,\tau(A)] \\ 0 & \text{else} \end{cases} \nonumber \\  f_B(t) & = & \begin{cases} 1 & t \in [1-\tau(B),1] \\ 0 & \text{else} \end{cases} \nonumber \end{eqnarray} and \[ f_C(t) = \begin{cases} 1 & t \in [0,\tau(C)] \\ 0 & \text{else} \end{cases} \] respectively. We conclude that $u_{qc}(\vec{y}, \vec{x}) = u_{loc}(\vec{y}, \vec{x})$ in this case. The cases when either $b$ or $c$ equal zero are identical, and the case when two of $a,b,c$ are zero is similar. It follows that every extreme point of $D_{qc}(3)$ is an element of $D_q(3)$. The final statement of the theorem follows from the fact that $D_{qc}(3)$ is a closed set. \end{proof}

\section*{Acknowledgements} This work was partly carried out while the author attended the workshop ``The Many Faceted Connes' Embedding Problem'' held at the Banff International Research Station. The author thanks Vern Paulsen for sharing Corollary \ref{OzawaCor} and Samuel Harris for pointing out Theorem 1.9.5 of his Ph.D. thesis \cite{HarrisThesis}.

\bibliographystyle{plain}
\bibliography{references}

\begin{thebibliography}{10}

\bibitem{CqNeqCqs}
A.~Coladangelo and J.~Stark.
\newblock Unconditional separation of finite and infinite-dimensional quantum
  correlations.
\newblock {\em arxiv}, abs/1708.06522, 2018.

\bibitem{ConnesConjecture}
A.~Connes.
\newblock Classification of injective factors cases ii$_1$, ii$_\infty$,
  iii$_\lambda$, $\lambda \neq 1$.
\newblock {\em Annals of Mathematics}, 104(1):73--115, 1976.

\bibitem{MR1402012}
K.~Davidson.
\newblock {\em {$C^*$}-algebras by example}, volume~6 of {\em Fields Institute
  Monographs}.
\newblock American Mathematical Society, Providence, RI, 1996.

\bibitem{DykemaPaulsen2016}
K.~Dykema and V.~Paulsen.
\newblock Synchronous correlation matrices and connes' embedding conjecture.
\newblock {\em J. Math. Phys.}, 57(1):015214, 12, 2016.

\bibitem{DeltaGame}
K.~Dykema, V.~Paulsen, and J.~Prakash.
\newblock The delta game.
\newblock {\em Quantum Information {\&} Computation}, 18(7{\&}8):599--616,
  2018.

\bibitem{DykemaPaulsenPrakash2019}
K.~Dykema, V.~Paulsen, and J.~Prakash.
\newblock Non-closure of the set of quantum correlations via graphs.
\newblock {\em Comm. Math. Phys.}, 365(3):1125--1142, 2019.

\bibitem{FritzKirchberg}
T.~Fritz.
\newblock Tsirelson's problem and kirchberg's conjecture.
\newblock {\em Reviews in Mathematical Physics}, 24(05):1250012, 2012.

\bibitem{PhysRevA.97.022104}
K.~Goh, J.~Kaniewski, E.~Wolfe, T.~V\'ertesi, X.~Wu, Y.~Cai, Y.~Liang, and
  V.~Scarani.
\newblock Geometry of the set of quantum correlations.
\newblock {\em Phys. Rev. A}, 97:022104, Feb 2018.

\bibitem{HarrisThesis}
{Harris, Samuel}.
\newblock Unitary correlation sets and their applications, 2019.

\bibitem{JMPPSW2011}
M.~Junge, M.~Navascues, C.~Palazuelos, D.~Perez-Garcia, V.~Scholz, and
  R.~Werner.
\newblock Connes' embedding problem and tsirelson's problem.
\newblock {\em J. Math. Phys.}, 52(1):012102, 2011.

\bibitem{KimPaulsenSchafhauser2018}
S.~Kim, V.~Paulsen, and C.~Schafhauser.
\newblock A synchronous game for binary constraint systems.
\newblock {\em J. Math. Phys.}, 59(3):032201, 17, 2018.

\bibitem{Musat2019}
M.~Musat and M.~R{\o}rdam.
\newblock Non-closure of quantum correlation matrices and factorizable channels
  that require infinite dimensional ancilla (with an appendix by {N}arutaka
  {O}zawa).
\newblock {\em Comm. Math. Phys.}, Apr 2019.

\bibitem{MR3067294}
N.~Ozawa.
\newblock About the {C}onnes embedding conjecture: algebraic approaches.
\newblock {\em Jpn. J. Math.}, 8(1):147--183, 2013.

\bibitem{OzawaUnpublished}
N.~Ozawa.
\newblock Around connes' embedding conjecture.
\newblock Lectures at conference ``Analysis in Quantum Information Theory'',
  Institut Henri Poincare, 2017.

\bibitem{MR3460238}
V.~Paulsen, S.~Severini, D.~Stahlke, I.~Todorov, and A.~Winter.
\newblock Estimating quantum chromatic numbers.
\newblock {\em J. Funct. Anal.}, 270(6):2188--2222, 2016.

\bibitem{SynchronousGeometry}
T.~Russell.
\newblock Geometry of the set of synchronous quantum correlations.
\newblock {\em Arxiv}, 2019.

\bibitem{Slofstra1}
W.~Slofstra.
\newblock The set of quantum correlations is not closed.
\newblock {\em Forum Math. Pi}, 7:e1, 41, 2019.

\end{thebibliography}
\end{document}